\documentclass[leqno,10pt]{amsart}
\usepackage[top=1.30in, bottom=1.30in, left=1.22in, right=1.22in]{geometry}
\usepackage{amscd,amsfonts,mathrsfs,amsthm,enumerate}
\usepackage{amssymb, amsmath,bigints}
\usepackage{stmaryrd}
\usepackage{epsfig}
\usepackage[sorted-cites,nobysame]{amsrefs}
\usepackage{latexsym}
\usepackage[usenames,dvipsnames]{color}
\usepackage[colorlinks=true,linkcolor=blue,citecolor=BrickRed,urlcolor=RoyalBlue]{hyperref}
\newtheorem{thm}{Theorem}[section]
\newtheorem{proposition}[thm]{Proposition}
\newtheorem{corollary}[thm]{Corollary}
\newtheorem{definition}[thm]{Definition}
\newtheorem{theorem}[thm]{Theorem}
\newtheorem{lemma}[thm]{Lemma}
\newtheorem{remark}[thm]{Remark}

\title{Expanding solutions of quasilinear parabolic equations}
\author{Nikolaos Roidos}
\address{Department of Mathematics, University of Patras, 26504 Rio Patras, Greece}
\email{roidos@math.upatras.gr}

\begin{document}

\date{\today}
\subjclass[2010]{35C20; 35K59; 35K65; 35K91; 35R01; 76S99}
\thanks{The author was supported by Deutsche Forschungsgemeinschaft, grant SCHR 319/9-1}

\begin{abstract}
By using the theory of maximal $L^{q}$-regularity and methods of singular analysis, we show a Taylor's type expansion--with respect to the geodesic distance around an arbitrary point--for solutions of quasilinear parabolic equations on closed manifolds. The powers of the expansion are determined explicitly by the local geometry, whose reflection to the solutions is established through the local space asymptotics.
\end{abstract}

\maketitle

\section{Introduction}

Let $X_{1}\overset{d}{\hookrightarrow} X_{0}$ be a continuously and densely injected complex Banach couple, $q\in(1,\infty)$, $U$ an open subset of the real interpolation space $(X_{1},X_{0})_{\frac{1}{q},q}$, $A(\cdot): U\rightarrow \mathcal{L}(X_{1},X_{0})$ a family of bounded operators and $F(\cdot,\cdot): U\times [0,T_{0}]\rightarrow X_{0}$, for some $T_{0}>0$. Consider the problem
\begin{eqnarray}\label{aqpp1}
u'(t)+A(u(t))u(t)&=&F(u(t),t)+G(t),\quad t\in(0,T),\\\label{aqpp2}
u(0)&=&u_{0},
\end{eqnarray}
where $T\in(0,T_{0}]$, $u_{0}\in U$ and $G\in L^{q}(0,T_{0};X_{0})$. Moreover, denote by $W^{1,q}$ the usual Sobolev space and let the following assumptions:\\
{\bf (H1)}{\em There exists a $T\in(0,T_{0})$ and a $u\in W^{1,q}(0,T;X_{0})\cap L^{q}(0,T;X_{1})$ solving \eqref{aqpp1}-\eqref{aqpp2}.}\\
{\bf (H2)} {\em $A(u(\cdot))\in C([0,T];\mathcal{L}(X_{1},X_{0}))$.}\\
{\bf (H3)} {\em $X_{0}$ is UMD and there exists a $\theta\in(\frac{\pi}{2},\pi)$ with the following property: for each $\xi\in [0,T]$ there exists a $c\geq 0$ such that the operator $A(u(\xi))+c: X_{1}\rightarrow X_{0}$ is $R$-sectorial of angle $\theta$.}\\
{\bf (H4)} $A(\cdot)$ is constant, i.e. the equation \eqref{aqpp1} is semilinear, $G=0$ and $F(u(\cdot),\cdot)\in C([0,T];(X_{0},X_{1})_{\phi,p})$ for some $\phi\in(0,1)$ and $p\in(1,\infty)$.\\
Denote
$$
\Lambda_{r}(\phi)=\{r+\lambda \,|\, \lambda\in\mathbb{C}\backslash\{0\}, |\arg(\lambda)|<\phi\}, \quad r\geq0, \quad \phi\in(0,\pi].
$$
Furthermore, if $Y$ is a Banach space and $\Omega\subseteq\mathbb{C}$ is an open set, denote by $\mathcal{A}(\Omega;Y)$ the space of $Y$-valued functions that are analytic in $\Omega$. In the first part of this paper we show the following:

\begin{theorem}\label{expth}
Let the assumptions {\em(H1)}, {\em(H2)}, {\em(H3)} be satisfied and let the family 
\begin{eqnarray}\nonumber
\lefteqn{[0,T)\ni \tau\mapsto v_{\tau}(\cdot)=e^{(\tau-(\cdot))A(u(\tau))}u(\tau)\in W^{1,q}(\tau,T;X_{0})\cap L^{q}(\tau,T;X_{1})}\\\label{vtaufamily}
&&\cap \, C([\tau,T];(X_{1},X_{0})_{\frac{1}{q},q})\cap \bigcap_{k\in\mathbb{N}}\mathcal{A}(\Lambda_{\tau}(\theta-\pi/2);\mathcal{D}((A(u(\tau)))^{k})),
\end{eqnarray}
where 
$$
\Lambda_{0}(\theta-\pi/2)\cup\{0\}\ni z\rightarrow e^{-zA(u(\tau))}\in\mathcal{L}(X_{0})
$$
is the analytic semigroup generated by $-A(u(\tau))$. Then, for any $0\leq t_{1}<t_{2}\leq T$, we have
\begin{gather}\label{inequv}
\|u-v_{t_{1}}\|_{W^{1,q}(t_{1},t_{2};X_{0})\cap L^{q}(t_{1},t_{2};X_{1})}+\|u-v_{t_{1}}\|_{C([t_{1},t_{2}];(X_{1},X_{0})_{\frac{1}{q},q})}<C f(t_{1},t_{2}),
\end{gather}
where 
$$
f(t_{1},t_{2})=\|A(u(t_{1}))u(\cdot)-A(u(\cdot))u(\cdot)+F(u(\cdot),\cdot)+G(\cdot)\|_{L^{q}(t_{1},t_{2};X_{0})}<\infty,
$$
and $C\geq1$ is only determined by $u$, $T$ and $q$; in particular for each $\varepsilon>0$ there exists some $\delta>0$ such that $t_{2}-t_{1}<\delta$ implies $f(t_{1},t_{2})<\varepsilon$.\\ 
If in addition {\em (H4)} is satisfied, then 
\begin{gather}\label{contreg}
u\in C^{1}((0,T];X_{0})\cap C((0,T];X_{1})
\end{gather}
and for each $\alpha\in(0,\phi)$ we also have 
\begin{gather}\label{diffineq}
\|(u-v_{t_{1}})(t_{2})\|_{X_{1}}<M(t_{2}-t_{1})^{\alpha}\|F(u(\cdot),\cdot)\|_{C([t_{1},t_{2}];(X_{0},X_{1})_{\phi,p})},
\end{gather}
for certain $M>0$ that only depends on $\phi$, $p$ and $\alpha$.
\end{theorem}

The assumption (H1) above is the existence of a {\em maximal $L^{q}$-regular} solution of \eqref{aqpp1}-\eqref{aqpp2}, i.e. a classical solution in the $L^{q}$-setting; we refer to \cite[Theorem 2.1]{CL} and \cite[Theorem 5.1.1]{PrSi} for this property. The assumption (H3) is the {\em maximal $L^{q}$-regularity} property for each of the operators $A(u(\xi))$, $\xi\in[0,T]$, i.e. the fact that an abstract first order linear non-homogeneous Cauchy problem for $A(u(\xi))$ has a classical solution as in (H1), see Section 2 for details.

Theorem \ref{expth} simply states that under suitable maximal $L^q$-regularity assumptions, after any time $t_{1}$ the solution $u$ of \eqref{aqpp1}-\eqref{aqpp2} can be approximated by the analytic solution of the abstract homogeneous {\em ``heat equation"} obtained by the $u(t_{1})$-linearization of the quasilinear operator. The approximation can take place around any $t\in(0,T)$ and the error term can become arbitrary small in the maximal $L^q$-regularity space norm uniformly in $t$ by choosing the time step $t_{2}-t_{1}$ sufficiently small. 

Although an approximation of $u$ with similar regularity as in \eqref{vtaufamily} can be obtained by a mollification sequence and a density argument, in our case $v_{t_{1}}(\cdot)$ is analytic and satisfies $v_{t_{1}}(t_{1})=u(t_{1})$. Hence, in particular, we have the following:

\begin{corollary}\label{aproxfinitpoints}
Let the assumptions {\em(H1)}, {\em(H2)}, {\em(H3)} be satisfied and let the family $v_{\tau}$, $\tau\in[0,T)$, from \eqref{vtaufamily}. For every $\varepsilon>0$ there exists an $N\in\mathbb{N}$ such that for each finite set of points $t_{1}<...<t_{n}$, $n\geq N$, in $[0,T]$ containing $\{\frac{kT}{N}\}_{k\in\{0,...,N\}}$ and the function $v:[0,T]\rightarrow X_{0}$ defined by $v(t)=v_{t_{j}}(t)$, $t\in[t_{j},t_{j+1})$, $j\in\{1,...,n-1\}$, the following holds:\\
{\bf (i)} $v$ restricted to each $[t_{j},t_{j+1})$, $j\in\{1,...,n-1\}$, satisfies the regularity \eqref{vtaufamily} with $\tau=t_{j}$.\\
{\bf (ii)} If $u$ is the solution of \eqref{aqpp1}-\eqref{aqpp2} according to the assumption {\em(H1)}, then 
$$
\|u-v\|_{W^{1,q}(t_{j},t_{j+1};X_{0})\cap L^{q}(t_{j},t_{j+1};X_{1})}+\|u-v\|_{C([t_{j},t_{j+1}];(X_{1},X_{0})_{\frac{1}{q},q})}<\varepsilon
$$
for each $j\in\{1,...,n-1\}$. If in addition {\em (H4)} is satisfied, then
$$
\|(u-v)(t)\|_{X_{1}}<M(t-t_{j})^{\alpha}\|F(u(\cdot),\cdot)\|_{C([t_{j},t];(X_{0},X_{1})_{\phi,p})}, \quad t\in[t_{j},t_{j+1}),
$$ 
for each $j\in\{1,...,n-1\}$, where $M$ is as in \eqref{diffineq}. \\
{\bf (iii)} $v(t_{j})=u(t_{j})$ for each $j\in\{1,...,n-1\}$.
\end{corollary}

The study of linearizations is one of the central ideas in the theory of partial differential equations. In order to avoid the large amount of literature, by focusing on the functional analytic approach to evolution equations, we mention certain examples where the linearization has been used. Concerning abstract non-autonomous linear parabolic problems we refer to \cite[Chapter II]{Am}, \cite[Chapter 6]{Lun} and \cite[Chapter 5]{Tan}. For general quasilinear parabolic problems we refer to \cite{Am3}, \cite{CL}, \cite{PrSi} and \cite[Chapter 6]{Tan}. In the above examples, the linearization was employed in order to obtain existence, uniqueness and regularity results for the solutions, e.g. for the construction of the evolution operator, for setting a Banach fixed point argument, for showing better regularity or global existence for the solution etc. More precisely, the solution of the linearized problem was considered as an approximation of the original solution in the following two ways: (i) in a finite set of time steps $t\in [0,T]$, through the construction of the evolution operator in the case of abstract non-autonomous linear parabolic problems, (ii) in $t=0$, by determining the centre of a ball where the contraction mapping was defined through the Banach fixed point argument in the general quasilinear case. However, such an approximation was neither considered in every $t\in[0,T]$, nor uniformly in $t$. Moreover, it was not regarded as a regularity improvement of the original solution and no results in this direction were obtained.

Besides its good regularity, we choose \eqref{expth} as an approximation due to the following reason. The maximal $L^{q}$-regularity space norm of the error term can be estimated around any $t\in(0,T)$ by using a standard maximal $L^{q}$-regularity inequality (see, e.g., \eqref{maxreqineq} below); i.e. by using the well posedness of the linearized problem due to (H3). Then, the continuity assumption (H2) turns out to be sufficient for this estimate to be uniform in $t$. This is the main difficulty in the above approach; it is equivalent to provide uniform estimates for several functional analytic constants and bounds related to our linearized family of operators (see Lemma \ref{MMA} below). For this reason, the theory of maximal $L^{q}$-regularity is extensively used through the document.

The question of whether an operator satisfies the property of maximal $L^{q}$-regularity was a long standing problem. We recall that if an operator satisfies this property, then necessarily it has to generate an analytic semigroup, see \cite[Theorem 2.2]{Dore}. First in \cite{DS} it was shown that in Hilbert spaces all generators of analytic semigroups have maximal $L^{q}$-regularity. Next, in \cite{DV} the result was extended to UMD Banach spaces (see Section 2 for the property of UMD) for operators having bounded imaginary powers. Then, in \cite{KaW} and \cite{Weis} a characterization of this property was given in UMD spaces through the $R$-sectoriality property for operators. We mention that the property of maximal $L^{q}$-regularity is deeply related to the old problem of the closedness of the sum of two closed operators, i.e. the question of whether the sum of two closed operators is again closed. For the commutative case this problem was treated first in \cite{DG}, and then in \cite{DV} and \cite{KaW}. For further details and an extensive study on the above two problems, we refer to \cite[Chapter III.4]{Am}, \cite[Chapter II]{PrSi} and to the references therein.

Next we focus on manifolds with isolated conical singularities. Such a space is a smooth compact $(n+1)$-dimensional manifold $\mathcal{B}$, $n\ge 1$, with closed possibly disconnected smooth boundary $\partial\mathcal{B}$ of dimension $n$. $\mathcal{B}$ is endowed with a degenerate Riemannian metric $g$ which in a collar neighborhood $[0,1)\times\partial\mathcal{B}$ of the boundary is of the form
\begin{gather}\label{metric}
g= dx^2+x^2h(x),
\end{gather}
where $[0,1)\ni x\mapsto h(x)$ is a smooth up to $x=0$ family of non-degenerate Riemannian metrics on the cross section $\partial\mathcal{B}$. The boundary $\{0\}\times \partial\mathcal{B}$ of $\mathcal{B}$ corresponds to the conical tips. We denote $\mathbb{B}=(\mathcal{B},g)$, $\partial\mathbb{B}=(\partial\mathcal{B},h(0))$ and $\partial\mathcal{B}=\cup_{k=1}^{\ell}\partial\mathcal{B}_{k}$, for certain $\ell\in\mathbb{N}$, where each of $\partial\mathcal{B}_{k}$ is closed, smooth and connected. Moreover, let $\mathbb{B}^{\circ}$ be the interior of $\mathbb{B}$.

The naturally appearing differential operators on $\mathbb{B}$, including the Laplacian, belong to the class of {\em cone differential operators}, which is a particular class of degenerate differential operators on $\mathcal{B}$. If $A$ is such an operator, we consider closed extensions of $A$ in {\em weighted Mellin-Sobolev spaces} $\mathcal{H}_{p}^{s,\gamma}(\mathbb{B})$, $p\in(1,\infty)$, $s,\gamma\in\mathbb{R}$, i.e. in the cone analogous of the usual Sobolev spaces. In the situation of $A$ being $\mathbb{B}$-elliptic, the domain of its maximal extension $\mathcal{D}(\underline{A}_{s,\max})$ differs from the domain of its minimal extension (i.e. its closure) $\mathcal{D}(\underline{A}_{s,\min})$ by an $(s,p)$-independent finite dimensional space $\mathcal{E}_{A,\gamma}$, i.e. we have 
\begin{gather}\label{dommaxexp}
\mathcal{D}(\underline{A}_{s,\max})=\mathcal{D}(\underline{A}_{s,\min})\oplus\mathcal{E}_{A,\gamma},
\end{gather}
see Section 3 for details. Hence, in general we have several closed extensions.

For the minimal domain we have $\mathcal{D}(\underline{A}_{s,\min})\hookrightarrow \mathcal{H}_{p}^{s+\mu,\gamma+\mu-\varepsilon}(\mathbb{B})$ for all $\varepsilon>0$, where $\mu$ stands for the order of the operator. The space $\mathcal{E}_{A,\gamma}$, called {\em asymptotics space}, consists of linear combinations of $C^{\infty}(\mathbb{B}^{\circ})$-functions that vanish on $\mathcal{B}\backslash([0,1)\times\partial\mathcal{B})$ and in local coordinates $(x,y)$ on the collar part $(0,1)\times\partial\mathcal{B}$ they are of the form $c(y)\omega(x)x^{-\rho}\log^{\eta}(x)$, where $c\in C^{\infty}(\partial\mathbb{B})$, $\omega:[0,1)\rightarrow [0,1]$ is a fixed cut-off function near zero, $\rho\in \mathbb{C}$ and $\eta\in\mathbb{N}\cup\{0\}$. Here the exponents $\rho$ and $\eta$ are determined explicitly by the coefficients of $A$ near the singularities, for instance if $A$ is the Laplacian induced by $g$, then they are only determined by the family $h(\cdot)$ and if in addition $h(\cdot)=h$ is constant, then they are only determined by the spectrum of the Laplacian on $\partial\mathbb{B}$. Hence, in particular, we have the following:

\begin{remark}\label{Remarkexp}
{\bf (i)} For each $k\in\mathbb{N}$ and each $\varepsilon>0$ we have $\mathcal{D}(\underline{A}_{s,\min}^{k})\hookrightarrow \mathcal{H}_{p}^{s+k\mu,\gamma+k\mu-\varepsilon}(\mathbb{B})$, so that, due to standard Mellin-Sobolev embeddings and under appropriate choice of the weight $\gamma$, elements in $\mathcal{D}(\underline{A}_{s,\min}^{k})$ vanish on the conical tips like $x^{\gamma+k\mu-\frac{n+1}{2}-\varepsilon}$. On the other hand the elements of the basis of $\mathcal{E}_{A^{k},\gamma}$ increase in number as $k$ increases: more terms, which are of the form $c(y)x^{-\rho}\log^{\eta}(x)$ close to the tips with smaller $\mathrm{Re}(\rho)$, are added. In the case of $A$ being the Laplacian on $\mathbb{B}$, we refer to \eqref{DD22} for a detailed description of the above integer powers domains.\\
Moreover, recall the following three facts: \\
{\bf (ii)} Under appropriate choice of the weight $\gamma$ we can have $R$-sectoriality for a non-minimal closed extension of a cone differential operator, see, e.g., Theorem \ref{RsecD} below for the case of the Laplacian on $\mathbb{B}$. In particular, in this case the domains of the integer powers of the operator will be as in {\em (i)}. \\
{\bf (iii)} By choosing geodesic polar coordinates we can regard a differential operator on a smooth closed manifold without boundary as a cone differential operator on a conic manifold $\mathbb{B}$, see Section 4.3. \\
{\bf (iv)} When we apply Theorem \ref{expth} or Corollary \ref{aproxfinitpoints} to a problem of the form \eqref{aqpp1}-\eqref{aqpp2} defined on a conic manifold, by \eqref{vtaufamily} the approximation of the solution $u$ on $[t_{1},t_{2})$ belongs to the domain of the $k$-th power of the $u(t_{1})$-linearization of the quasilinear operator, for each $k\in\mathbb{N}$. As a consequence, when this linearization is a cone differential operator, the approximation can be expressed as in {\em(i)}. The same result applies to smooth closed manifolds without boundary due to {\em(iii)}.
\end{remark}

The Remark \ref{Remarkexp} above combined with Theorem \ref{expth} or Corollary \ref{aproxfinitpoints}, provides locally in space and time expansions for the solutions of quasilinear parabolic equations on conic manifolds or on smooth closed manifolds. These expansions consist of three terms. The first term is an asymptotics space $\mathcal{E}_{A^{k},\gamma}$ part, i.e a sum of terms of the form $c(y)x^{-\rho}\log^{\eta}(x)$, which can be chosen arbitrary long. The parameter $x$ here stands for the geodesic distance. The second term, induced by a $\mathcal{D}(\underline{A}_{s,\min}^{k})$ term, belongs to a Mellin-Sobolev space and close to the conical tips (or close to the points where the expansion takes place) it decays to zero faster than each of the non-constant summands of the previous term. The third term is a remainder appearing by the nonlinearity and can be chosen arbitrary small with respect to the maximal $L^{q}$-regularity norm and in particular with respect to the $C^{0}$-norm, due to \eqref{inequv} or \eqref{diffineq}.

The above expansion provides information on the asymptotic behavior of the solutions close to the singularities or close to an arbitrary point where polar coordinates are chosen. Besides this, the reflection of the local geometry to the evolution is established through the exponents $\rho$, $\eta$ and the choice of the weight $\gamma$. As an application, in Section 4 we consider two nonlinear problems on conic manifolds, namely, the porous medium equation and the Swift-Hohenberg equation. In each of these cases we obtain the expansion described above and provide a comprehensive description of the asymptotics space part in terms of the family of metrics $h(\cdot)$.

\section{Decomposing the solutions of the abstract quasilinear parabolic problem}

We start with some elementary facts from the linear theory. Recall that $X_{1}\overset{d}{\hookrightarrow} X_{0}$ denotes a continuously and densely injected complex Banach couple. Moreover, recall that a linear operator $A$ in $X_{0}$ generates an analytic semigroup if and only if there exists a positive shift of $A$ that is {\em sectorial} of angle $\pi/2$, see, e.g., \cite[Corollary 3.7.17]{ABHN}. 

\begin{definition}[Sectorial operators]\label{Sec}
Let $\mathcal{P}(K,\theta)$, $\theta\in[0,\pi)$, $K\geq1$, be the class of all closed densely defined linear operators $A$ in $X_{0}$ such that 
$$
S_{\theta}=\{\lambda\in\mathbb{C}\,|\, |\arg(\lambda)|\leq\theta\}\cup\{0\}\subset\rho{(-A)} \quad \mbox{and} \quad (1+|\lambda|)\|(A+\lambda)^{-1}\|_{\mathcal{L}(X_{0})}\leq K
$$
when $ \lambda\in S_{\theta}$. The elements in $\mathcal{P}(\theta)=\cup_{K\geq1}\mathcal{P}(K,\theta)$ are called {\em (invertible) sectorial operators of angle $\theta$}. If $A\in\mathcal{P}(\theta)$ then any $K\geq1$ such that $A\in \mathcal{P}(K,\theta)$ is called {\em sectorial bound of $A$}, and we denote $K_{A,\theta}=\inf\{K\, |\, A\in \mathcal{P}(K,\theta)\}$ depending on $A$ and $\theta$. 
\end{definition}

Consider the linear problem
\begin{eqnarray}\label{app1}
u'(t)+Au(t)&=&f(t), \quad t\in(0,T),\\\label{app2}
u(0)&=&u_{0},
\end{eqnarray}
where $T>0$, $u_{0}\in(X_{1},X_{0})_{\frac{1}{q},q}$, $q\in(1,\infty)$, $f\in L^{q}(0,T;X_{0})$ and $-A:\mathcal{D}(A)=X_{1}\rightarrow X_{0}$ is the infinitesimal generator of an analytic semigroup on $X_{0}$. We say that $A$ has {\em maximal $L^q$-regularity} if for any $u_{0}\in(X_{1},X_{0})_{\frac{1}{q},q}$ and any $f\in L^{q}(0,T;X_{0})$ there exists a unique $u\in W^{1,q}(0,T;X_{0})\cap L^{q}(0,T;X_{1})$ solving \eqref{app1}-\eqref{app2}. In this case, $u$ also depends continuously on $u_{0}$ and $f$, see, e.g., \cite[(2.2)]{CL}. Furthermore, the above property is independent of $T$ and $q$, see \cite[Theorem 2.5 and Theorem 4.2]{Dore}, and for any $0\leq\tau_{1}<\tau_{2}\leq T<\infty$ the following embedding holds
\begin{gather}\label{intemb}
W^{1,q}(\tau_{1},\tau_{2};X_{0})\cap L^{q}(\tau_{1},\tau_{2};X_{1})\hookrightarrow C([\tau_{1},\tau_{2}];(X_{1},X_{0})_{\frac{1}{q},q}),
\end{gather}
see, e.g., \cite[Theorem III.4.10.2]{Am}; note that if we restrict to the subspace of $W^{1,q}(\tau_{1},\tau_{2};X_{0})\cap L^{q}(\tau_{1},\tau_{2};X_{1})$ consisting of functions $u$ satisfying $u(0)=0$, then the norm of the above embedding is independent of $\tau_{1},\tau_{2}\in[0,T]$, see \cite[Corollary 2.3]{CL}. In addition, in order an operator $A$ to have maximal $L^{q}$-regularity, it is sufficient to satisfy the above property for $u_{0}=0$ only. Finally, note that after replacing $u$ with $e^{ct}u$ in \eqref{app1}-\eqref{app2}, we can insert a positive shift $c>0$ to the operator $A$. 

By integrating \eqref{app1} from $0$ to $t\in(0,T]$, we can easily verify that if $A$ has maximal $L^{q}$-regularity, then $u$ is also a {\em mild solution} to \eqref{app1}-\eqref{app2} in the sense of \cite[Section 3.1]{ABHN}. Therefore, if $\{e^{-tA}\}_{t\geq 0}$ is the semigroup generated by $-A$, then the solution is expressed by the {\em variation of constants formula} as
\begin{gather}\label{solution}
u(t)=e^{-tA}u_{0}+\int_{0}^{t}e^{(s-t)A}f(s)ds, \quad t\in[0,T],
\end{gather}
see, e.g., \cite[Proposition 3.1.16]{ABHN}. Moreover, recall that for any $k\in\mathbb{N}$, $k\geq2$, the integer powers $A^{k}:\mathcal{D}(A^{k})\rightarrow X_{0}$ of $A$ satisfy 
\begin{gather}\label{DAk}
\mathcal{D}(A^{k})=\{x\in \mathcal{D}(A^{k-1})\, |\, Ax\in \mathcal{D}(A^{k-1})\} \quad \text{with} \quad \|\cdot\|_{\mathcal{D}(A^{k})}=\sum_{j=0}^{k}\|A^{j}\cdot\|_{X_{0}},
\end{gather}
where $A^{0}=I$.

Since all Banach spaces we consider in the sequel belong to the class of UMD, i.e. they satisfy the {\em unconditionality of martingale differences property}, see \cite[Section III.4.4]{Am}, we recall the following boundedness condition for the resolvent of an operator, which is related to the maximal $L^{q}$-regularity property in these spaces.

\begin{definition}[$R$-sectorial operators]\label{RSec}
Denote by $\mathcal{R}(K,\theta)$, $\theta\in[0,\pi)$, $K\geq1$, the class of all operators $A\in \mathcal{P}(\theta)$ in $X_{0}$ such that for any choice of $\lambda_{1},...,\lambda_{N}\in S_{\theta}\backslash\{0\}$ and $x_{1},...,x_{N}\in X_0$, $N\in\mathbb{N}$, we have
$$
\Big\|\sum_{k=1}^{N}\epsilon_{k}\lambda_{k}(A+\lambda_{k})^{-1}x_{k}\Big\|_{L^{2}(0,1;X_0)} \leq K \Big\|\sum_{k=1}^{N}\epsilon_{k}x_{k}\Big\|_{L^{2}(0,1;X_0)},
$$
where $\{\epsilon_{k}\}_{k\in\mathbb{N}}$ is the sequence of the Rademacher functions. The elements in $\mathcal{R}(\theta)=\cup_{K\geq1}\mathcal{R}(K,\theta)$ are called {\em $R$-sectorial operators of angle $\theta$}. If $A\in \mathcal{R}(\theta)$ then any $K\geq1$ such that $A\in \mathcal{R}(K,\theta)$ is called {\em $R$-sectorial bound of $A$}, and we denote $R_{A,\theta}=\inf\{K\, |\, A\in \mathcal{R}(K,\theta)\}$ depending on $A$ and $\theta$.
\end{definition} 

\noindent The following classical result holds. 

\begin{theorem}[{\rm Kalton and Weis, \cite[Theorem 6.5]{KaW}}]\label{KWth}
If $X_{0}$ is UMD and $A\in\mathcal{R}(\theta)$ in $X_{0}$ with $\theta>\frac{\pi}{2}$, then $A$ has maximal $L^{q}$-regularity. 
\end{theorem}

Let $\theta\in(\frac{\pi}{2},\pi)$, $c\geq0$ and $A:\mathcal{D}(A)=X_{1}\rightarrow X_{0}$ such that $A+c \in\mathcal{R}(\theta)$ with $R$-sectorial bound equal to $R$. If $X_{0}$ is UMD, then any solution $u$ of \eqref{app1}-\eqref{app2} satisfies 
\begin{eqnarray}\nonumber
\|u\|_{W^{1,q}(0,T;X_{0})\cap L^{q}(0,T;X_{1})}&=&\|u\|_{L^{q}(0,T;X_{0})}+\|u'\|_{L^{q}(0,T;X_{0})}+\|u\|_{L^{q}(0,T;X_{1})}\\\label{maxreqineq}
&\leq& e^{cT}M(\|f\|_{L^{q}(0,T;X_{0})}+\|u_{0}\|_{(X_{1},X_{0})_{\frac{1}{q},q}}),
\end{eqnarray}
for some constant $M\geq1$ that only depends on $\{A,c,R,\theta,q\}$, see, e.g., \cite[(2.11)]{Ro2} or \cite[(6.3)]{KaW}. Concerning the $\{A,c,R\}$-dependence of $M$, we have the following.

\begin{lemma}[Uniform boundedness]\label{MMA}
Let $q\in(1,\infty)$, $\theta\in(\frac{\pi}{2},\pi)$, $r>0$, $X_{0}$ be a UMD space and let $A(\cdot)\in C([0,r];\mathcal{L}(X_{1},X_{0}))$. Assume that for each $\xi\in [0,r]$ there exists a $c\geq0$ such that the operator $A(\xi)+c:X_{1}\rightarrow X_{0}$ is $R$-sectorial of angle $\theta$, and let $M\geq1$ satisfy \eqref{maxreqineq} with respect to $A(\xi)$. Then:\\
 {\bf (i)} There exists some $c_{0}\geq1$ such that for each $\xi\in[0,r]$ the operator $A(\xi)+c_{0}:X_{1}\rightarrow X_{0}$ is $R$-sectorial of angle $\theta$ with $R$-sectorial bound equal to $c_{0}$. \\
 {\bf (ii)} There exists some $M_{0}\geq1$ such that for each $\xi\in [0,r]$ we can choose $M\leq M_{0}$.\\
 {\bf (iii)} $c_{0}$ and $M_{0}$ only depend on $\{A(\cdot),r,\theta\}$ and $\{A(\cdot),r,\theta,q\}$ respectively.
\end{lemma}
\begin{proof} 
Let $\xi\in[0,r]$ and $c\geq0$ such that $A(\xi)+c\in \mathcal{R}(\theta)$. Let $K_{A(\xi)+c,\theta}$ and $R_{A(\xi)+c,\theta}$ be as in Definition \ref{Sec} and Definition \ref{RSec} respectively. 
Moreover, for each $\xi\in[0,r]$ denote
$$
c_{\xi,\theta} = \inf\{\rho\geq0 \, | \, \text{$A(\xi)+\rho : X_{1}\rightarrow X_{0}$ is $R$-sectorial of angle $\theta$}\}.
$$ 
Assume that there exists a sequence $\{\xi_{j}\}_{j\in\mathbb{N}}$ in $[0,r]$ such that $c_{\xi_{j},\theta}\rightarrow\infty$ as $j\rightarrow\infty$. By possibly passing to a subsequence, we may assume that $\xi_{j}\rightarrow \widetilde{\xi}\in[0,r]$ as $j\rightarrow\infty$. Let $\widetilde{c}\geq0$ such that $A(\widetilde{\xi})+\widetilde{c}\in \mathcal{R}(\theta)$. Choose $\delta>0$ such that
\begin{eqnarray}\label{firstbound}
\|(A(\widetilde{\xi})-A(\xi))(A(\widetilde{\xi})+\widetilde{c})^{-1}\|_{\mathcal{L}(X_{0})}<\frac{1}{2}\min\Big\{\frac{1}{1+K_{A(\widetilde{\xi})+\widetilde{c},\theta}},\frac{1}{1+R_{A(\widetilde{\xi})+\widetilde{c},\theta}}\Big\}
\end{eqnarray}
whenever $|\widetilde{\xi}-\xi|<\delta$, $\xi\in [0,r]$. Then, for such $\xi$ by the formula
\begin{eqnarray}\label{secondbound}
(A(\xi)+\widetilde{c}+\lambda)^{-1}=(A(\widetilde{\xi})+\widetilde{c}+\lambda)^{-1}\sum_{k=0}^{\infty}\big((A(\widetilde{\xi})-A(\xi))(A(\widetilde{\xi})+\widetilde{c}+\lambda)^{-1}\big)^{k}
\end{eqnarray}
valid for any $\lambda\in S_{\theta}$, we deduce that $S_{\theta}\subset\rho(-(A(\xi)+\widetilde{c}))$, $A(\xi)+\widetilde{c}\in\mathcal{P}(\theta)$ with sectorial bound $2K_{A(\widetilde{\xi})+\widetilde{c},\theta}$ and also $A(\xi)+\widetilde{c} \in \mathcal{R}(\theta)$ with $R$-sectorial bound $2R_{A(\widetilde{\xi})+\widetilde{c},\theta}$, so that we get a contradiction. Therefore, the set $\cup_{\xi\in[0,r]}\{c_{\xi,\theta}\}$ is bounded. By noting that $B\in \mathcal{R}(\theta)$ and $\nu\geq0$ implies $B+\nu \in \mathcal{R}(\theta)$ (see, e.g., \cite[Lemma 2.6]{RS3}), we conclude that there exists some $c_{0}\geq1$ such that for each $\xi\in[0,r]$ we have $A(\xi)+c_{0}\in\mathcal{R}(\theta)$.

Assume that there exists a sequence $\{\tau_{j}\}_{j\in\mathbb{N}}$ in $[0,r]$ such that $R_{A(\tau_{j})+c_{0},\theta}\rightarrow \infty$ as $j\rightarrow\infty$. By possibly passing to a subsequence, we may assume that $\tau_{j}\rightarrow \widetilde{\tau}\in[0,r]$ as $j\rightarrow\infty$. Let $\varepsilon>0$ such that $|\widetilde{\tau}-\tau|<\varepsilon$, $\tau\in [0,r]$, implies \eqref{firstbound} with $\{\widetilde{\xi},\xi,\widetilde{c}\}$ replaced by $\{\widetilde{\tau},\tau,c_{0}\}$. For any $\lambda\in S_{\theta}$ and $\tau$ as before, by \eqref{secondbound} with $\{\widetilde{\xi},\xi,\widetilde{c}\}$ replaced by $\{\widetilde{\tau},\tau,c_{0}\}$ we infer that $S_{\theta}\subset\rho(-(A(\tau)+c_{0}))$, $A(\tau)+c_{0}\in\mathcal{P}(\theta)$ with sectorial bound $2K_{A(\widetilde{\tau})+c_{0},\theta}$ and also $A(\tau)+c_{0} \in \mathcal{R}(\theta)$ with $R$-sectorial bound $2R_{A(\widetilde{\tau})+c_{0},\theta}$, which provides a contradiction. Hence, the set $\cup_{\xi\in[0,r]}\{R_{A(\xi)+c_{0},\theta}\}$ is bounded. Therefore, we conclude that there exists some $C_{0}\geq1$ such that for each $\xi\in[0,r]$ we have $A(\xi)+c_{0}\in\mathcal{R}(\theta)$ with $R$-sectorial bound equal to $C_{0}$. If $c_{0}< C_{0}$, then due to \cite[Lemma 2.6]{RS3} we can replace $c_{0}$ and $C_{0}$ by $C_{0}(1+\frac{2}{\sin(\theta)})$.

Let $T>0$ and for each $\xi\in [0,r]$ denote
$$
M_{\xi}=\inf \{M\, |\, \text{ \eqref{maxreqineq} holds with respect to $A(\xi)$ with the choice $c=R= c_{0}$}\}.
$$
Assume that there exists a sequence $\{\zeta_{j}\}_{j\in\mathbb{N}}$ in $[0,r]$ such that $M_{\zeta_{j}}\rightarrow \infty$ as $j\rightarrow \infty$. By possibly passing to a subsequence, we may assume that $\zeta_{j}\rightarrow \widetilde{\zeta}\in [0,r]$ as $j\rightarrow \infty$. Let $N\in\mathbb{N}$ such that $j\geq N$ implies 
$$
2e^{c_{0}T}M_{\widetilde{\zeta}}\|A(\widetilde{\zeta})-A(\zeta_{j})\|_{\mathcal{L}(X_{1},X_{0})}<1. 
$$
Let $u_{0}\in(X_{1},X_{0})_{\frac{1}{q},q}$, $f\in L^{q}(0,T;X_{0})$ and denote by $u_{\widetilde{\zeta}}$ and $u_{\zeta_{j}}$ the solution of \eqref{app1}-\eqref{app2} with respect to $A(\widetilde{\zeta})$ and $A(\zeta_{j})$ respectively. We have that
\begin{eqnarray*}
(u_{\widetilde{\zeta}}-u_{\zeta_{j}})'(t)+A(\widetilde{\zeta})(u_{\widetilde{\zeta}}-u_{\zeta_{j}})(t)&=&(A(\zeta_{i})-A(\widetilde{\zeta}))u_{\zeta_{j}}(t), \quad t\in(0,T),\\
(u_{\widetilde{\zeta}}-u_{\zeta_{j}})(0)&=&0.
\end{eqnarray*}
Therefore by \eqref{maxreqineq} applied to $A(\widetilde{\zeta})$ we estimate
\begin{eqnarray*}
\lefteqn{\|u_{\zeta_{j}}\|_{W^{1,q}(0,T;X_{0})\cap L^{q}(0,T;X_{1})}}\\
&\leq& \|u_{\widetilde{\zeta}}\|_{W^{1,q}(0,T;X_{0})\cap L^{q}(0,T;X_{1})}+\|u_{\widetilde{\zeta}}-u_{\zeta_{j}}\|_{W^{1,q}(0,T;X_{0})\cap L^{q}(0,T;X_{1})}\\
&\leq& e^{c_{0}T}M_{\widetilde{\zeta}}(\|f\|_{L^{q}(0,T;X_{0})})+\|u_{0}\|_{(X_{1},X_{0})_{\frac{1}{q},q}})\\
&&+e^{c_{0}T}M_{\widetilde{\zeta}}\|A(\widetilde{\zeta})-A(\zeta_{j})\|_{\mathcal{L}(X_{1},X_{0})}\|u_{\zeta_{j}}\|_{L^{q}(0,T;X_{1})},
\end{eqnarray*}
and a contradiction follows.
\end{proof}

\noindent
{\em Proof of Theorem \ref{expth}.}
Denote $B=A(u(t_{1}))$,
$$
Q(\cdot)=Bu(t_{1}+\cdot)-A(u(t_{1}+\cdot))u(t_{1}+\cdot)+F(u(t_{1}+\cdot),t_{1}+\cdot)+G(t_{1}+\cdot) \in L^{q}(0,t_{2}-t_{1};X_{0})
$$
and consider the problem
\begin{eqnarray}\label{AAaqpp11}
\eta'(t)+B\eta(t)&=&Q(t),\quad t\in(0,t_{2}-t_{1}),\\\label{AAaqpp22}
\eta(0)&=&u(t_{1}).
\end{eqnarray}
The above system has a solution in $W^{1,q}(0,t_{2}-t_{1};X_{0})\cap L^{q}(0,t_{2}-t_{1};X_{1})$ given by $u(t_{1}+\cdot)$. Moreover, (H1), (H3) and \eqref{intemb} imply the existence and uniqueness of $\eta\in W^{1,q}(0,t_{2}-t_{1};X_{0})\cap L^{q}(0,t_{2}-t_{1};X_{1})$ solving \eqref{AAaqpp11}-\eqref{AAaqpp22}. Therefore, $u(t)=\eta(t-t_{1})$ for each $t\in[t_{1},t_{2})$. As a consequence, by \eqref{solution} we deduce that $u(t)=v_{t_{1}}(t)+w_{t_{1}}(t)$, $t\in[t_{1},t_{2})$, where
$$
v_{t_{1}}(t)=e^{(t_{1}-t)B}u(t_{1}) \quad \text{and} \quad w_{t_{1}}(t)=\int_{0}^{t-t_{1}}e^{(t_{1}+s-t)B}Q(s)ds, \quad t\in[t_{1},t_{2}).
$$

According to \cite[(3.26)]{Tan} and \cite[Theorem 3.3.4]{Tan} there exists a unique 
$$
\rho\in C([0,\infty);X_{0})\cap C^{1}((0,\infty);X_{0})\cap C((0,\infty);\mathcal{D}(B))
$$
solving 
\begin{eqnarray*}
\rho'(t)+B\rho(t)&=&0,\quad t>0,\\
\rho(0)&=&u(t_{1}),
\end{eqnarray*}
which satisfies $\rho(t)=v_{t_{1}}(t_{1}+t)$, $t\geq0$. Due to maximal $L^q$-regularity of $B$ we also have
$$
v_{t_{1}}\in W^{1,q}(t_{1},T;X_{0})\cap L^{q}(t_{1},T;X_{1})\hookrightarrow C([t_{1},T];(X_{1},X_{0})_{\frac{1}{q},q}),
$$
where we have used \eqref{intemb}. Denote by $\Gamma_{\theta}$ the counterclockwise oriented boundary of the sector $S_{\theta}$. By \cite[Proposition 3.1.9 (i)]{ABHN} we have that $B+c$ generates an analytic semigroup which by \cite[3.46]{ABHN} or \cite[(3.26)]{Tan} is represented as
$$
e^{-z(B+c)}=\frac{1}{2\pi i}\int_{\Gamma_{\theta}}e^{z\lambda}(B+\lambda)^{-1}d\lambda, \quad z\in \Lambda_{0}(\theta-\pi/2),
$$
and in addition satisfies $e^{-zB}=e^{cz}e^{-z(B+c)}$, $z\in \Lambda_{0}(\theta-\pi/2)$, where $c$ is as in (H3). By Cauchy's theorem, see, e.g., the proof of \cite[Proposition 2.1.1]{Lun}, for each $k\in\mathbb{N}\cup\{0\}$ we have that 
$$
 \Lambda_{0}(\pi/2-\theta)\ni z\mapsto e^{-z(B+c)}=\frac{1}{2\pi i}\int_{\Gamma_{\theta}}(-\lambda)^{k}e^{z\lambda}(B+c)^{-k}(B+c+\lambda)^{-1}d\lambda\in \mathcal{L}(X_{0},\mathcal{D}(B^{k}))
$$
and moreover the above map is analytic. 

The difference $\eta-\rho$ satisfies
\begin{eqnarray}\label{diff1}
(\eta-\rho)'(t)+B(\eta-\rho)(t)&=&Q(t), \quad t\in(0,t_{2}-t_{1}),\\\label{diff2}
(\eta-\rho)(0)&=&0.
\end{eqnarray}
Therefore by the maximal $L^{q}$-regularity inequality \eqref{maxreqineq} we obtain 
$$
\|\eta-\rho\|_{W^{1,q}(0,t_{2}-t_{1};X_{0})\cap L^{q}(0,t_{2}-t_{1};X_{1})} \leq e^{c_{0}(t_{2}-t_{1})}M_{0}\|Q\|_{L^{q}(0,t_{2}-t_{1};X_{0})}
$$
and \eqref{inequv} follows by \eqref{intemb}; note that the data determining the spectral shift $c_{0}$ and the bound $M_{0}$ is described in Lemma \ref{MMA}.

If in addition (H4) is satisfied, then
$$
Q(\cdot)=F(u(t_{1}+\cdot),t_{1}+\cdot) \in C([0,t_{2}-t_{1}];(X_{0},X_{1})_{\phi,p}),
$$
so that, by extending $Q(\cdot)$ to $(t_{2}-t_{1},\infty)$ by constant, from \cite[(2.10)]{Ro2} applied to \eqref{AAaqpp11}-\eqref{AAaqpp22} we obtain \eqref{contreg}. Finally, \eqref{diffineq} is obtained by \cite[(2.9)]{Ro2} applied to \eqref{diff1}-\eqref{diff2}. \mbox{\ }\hfill$\square$

\section{Cone differential operators and the Laplacian on manifolds with conical singularities}

In this section we explain the parts (i) and (ii) in Remark \ref{Remarkexp}. We start by recalling some basic facts concerning the calculus of cone differential operators. For further details we refer to \cite{SS1}, \cite{SS} and to the references therein.

\begin{definition}[Cone differential operators]
A {\em cone differential operator} $A$ of order $\mu\in\mathbb{N}\cup\{0\}$ is an $\mu$-th order differential operator with smooth coefficients in $\mathbb{B}^{\circ}$ such that when it is restricted to the collar part $(0,1)\times\partial\mathcal{B}$ it admits the form 
$$
A=x^{-\mu}\sum_{k=0}^{\mu}a_{k}(x)(-x\partial_{x})^{k}, \quad \mbox{where} \quad a_{k}\in C^{\infty}([0,1);\mathrm{Diff}^{\mu-k}(\partial\mathcal{B})), \quad k\in\{0,...,\mu\}.
$$
\end{definition}

Let $(\xi,\zeta)$ be the covariables corresponding to the local coordinates $(x,y)\in [0,1)\times \partial\mathcal{B}$ near the boundary. Beyond its usual homogenous principal symbol $\sigma_{\psi}(A)\in C^{\infty}(T^{\ast}\mathbb{B}^{\circ}\backslash\{0\})$, the {\em rescaled symbol} of a cone differential operator $A$ is defined by 
$$
\widetilde{\sigma}_{\psi}(A)(y,\zeta,\xi)=\sum_{k=0}^{\mu}\sigma_{\psi}(a_{k})(0,y,\zeta)(-i\xi)^{k}\in C^{\infty}((T^{\ast}\partial\mathbb{B}\times\mathbb{R})\backslash\{0\}).
$$
Moreover, the {\em conormal symbol} of $A$ is defined to be the following family of differential operators on the boundary
$$
\mathbb{C}\ni \lambda \mapsto \sigma_{M}(A)(\lambda)=\sum_{k=0}^{\mu}a_{k}(0)\lambda^{k} \in \mathcal{L}(H_{2}^{\mu}(\partial\mathbb{B}),H_{2}^{0}(\partial\mathbb{B})),
$$
where $H_{p}^{s}(\partial\mathbb{B})$, $s\in\mathbb{R}$, $p\in(1,\infty)$, denotes again the usual Sobolev space. Note that due to \cite[(2.13)]{SS}, for a cone differential operator $B$ of order $\nu$ we have
$$
\sigma_{M}(AB)(\lambda)=\sigma_{M}(A)(\lambda+\nu)\sigma_{M}(B)(\lambda).
$$
The notion of ellipticity is extended to the case of conically degenerate differential operators as follows. 

\begin{definition}[$\mathbb{B}$-ellipticity]
A cone differential operator $A$ is called {\em $\mathbb{B}$-elliptic} if $\sigma_{\psi}(A)$ and $\widetilde{\sigma}_{\psi}(A)$ are pointwise invertible.
\end{definition}

We consider the cut-off function $\omega$ in the description of \eqref{dommaxexp} as a $C^{\infty}(\mathbb{B})$-function by extending it by zero. Moreover, denote by $C_{c}^{\infty}$ the space of smooth compactly supported functions. 

\begin{definition}[Mellin-Sobolev spaces]
For any $\gamma\in\mathbb{R}$ consider the map 
$$
M_{\gamma}: C_{c}^{\infty}(\mathbb{R}_{+}\times\mathbb{R}^{n})\rightarrow C_{c}^{\infty}(\mathbb{R}^{n+1}) \quad \mbox{defined by} \quad u(x,y)\mapsto e^{(\gamma-\frac{n+1}{2})x}u(e^{-x},y). 
$$
Furthermore, take a covering $\kappa_{j}:U_{j}\subseteq\partial\mathcal{B} \rightarrow\mathbb{R}^{n}$, $j\in\{1,...,N\}$, $N\in\mathbb{N}$, of $\partial\mathcal{B}$ by coordinate charts and let $\{\phi_{j}\}_{j\in\{1,...,N\}}$ be a subordinated partition of unity. For any $p\in(1,\infty)$ and $s,\gamma\in\mathbb{R}$ let $\mathcal{H}^{s,\gamma}_p(\mathbb{B})$ be the space of all distributions $u$ on $\mathbb{B}^{\circ}$ such that 
$$
\|u\|_{\mathcal{H}^{s,\gamma}_p(\mathbb{B})}=\sum_{j=1}^{N}\|M_{\gamma}(1\otimes \kappa_{j})_{\ast}(\omega\phi_{j} u)\|_{H^{s}_p(\mathbb{R}^{n+1})}+\|(1-\omega)u\|_{H^{s}_p(\mathbb{B})}
$$
is defined and finite, where $\ast$ refers to the push-forward of distributions. The space $\mathcal{H}^{s,\gamma}_p(\mathbb{B})$, called {\em (weighted) Mellin-Sobolev space}, is independent of the choice of the cut-off function $\omega$, the covering $\{\kappa_{j}\}_{j\in\{1,...,N\}}$ and the partition $\{\phi_{j}\}_{j\in\{1,...,N\}}$.
\end{definition}

Note that since the usual Sobolev spaces are UMD, by \cite[Theorem III.4.5.2]{Am}, the Mellin-Sobolev spaces are also UMD. Moreover, if $s\in \mathbb{N}\cup\{0\}$, then $\mathcal{H}^{s,\gamma}_p(\mathbb{B})$ is the space of all functions $u$ in $H^s_{p,loc}(\mathbb{B}^\circ)$ such that
$$
x^{\frac{n+1}2-\gamma}(x\partial_x)^{k}\partial_y^{\alpha}(\omega(x) u(x,y)) \in L_{loc}^{p}\big([0,1)\times \partial \mathcal{B}, \sqrt{\mathrm{det}[h(x)]}\frac{dx}xdy\big),\quad k+|\alpha|\le s.
$$

Cone differential operators act naturally on scales of weighted Mellin-Sobolev spaces, i.e. such an operator $A$ of order $\mu$ induces a bounded map
$$
A: \mathcal{H}^{s+\mu,\gamma+\mu}_p(\mathbb{B}) \rightarrow \mathcal{H}^{s,\gamma}_p(\mathbb{B}) \quad \text{for all} \quad p\in(1,\infty) \quad \text{and}\quad s,\gamma\in\mathbb{R}.
$$
However, we regard $A$ as an unbounded operator in $\mathcal{H}^{s,\gamma}_p(\mathbb{B})$, $p\in(1,\infty)$, $s,\gamma\in\mathbb{R}$, with domain $C_{c}^{\infty}(\mathbb{B}^{\circ})$. In the case of $A$ being $\mathbb{B}$-elliptic, the domain of its minimal extension $\underline{A}_{s,\min}$ is given by 
$$
\mathcal{D}(\underline{A}_{s,\min})=\Big\{u\in \bigcap_{\varepsilon>0}\mathcal{H}^{s+\mu,\gamma+\mu-\varepsilon}_p(\mathbb{B}) \, |\, Au\in \mathcal{H}^{s,\gamma}_p(\mathbb{B})\Big\}.
$$ 
If in addition the conormal symbol of $A$ is invertible on the line $\{\lambda\in\mathbb{C}\,|\, \mathrm{Re}(\lambda)= \frac{n+1}{2}-\gamma-\mu\}$, then we have precisely
$$
\mathcal{D}(\underline{A}_{s,\min})=\mathcal{H}^{s+\mu,\gamma+\mu}_p(\mathbb{B}).
$$ 

The domain of the maximal extension $\underline{A}_{s,\max}$ of $A$, defined as usual by
$$
\mathcal{D}(\underline{A}_{s,\max})=\Big\{u\in\mathcal{H}^{s,\gamma}_p(\mathbb{B}) \, |\, Au\in \mathcal{H}^{s,\gamma}_p(\mathbb{B})\Big\},
$$
is expressed as in \eqref{dommaxexp}. The set of exponents $\rho$ describing $\mathcal{E}_{A,\gamma}$ in \eqref{dommaxexp} coincides with the finite set of points $Q_{A,\gamma}=Z_{A}\cap I_{\mu,\gamma}$. Here 
\begin{gather}\label{strip}
I_{\mu,\gamma}=\Big\{\lambda\in\mathbb{C}\, |\, \mathrm{Re}(\lambda)\in \Big[\frac{n+1}{2}-\gamma-\mu,\frac{n+1}{2}-\gamma\Big)\Big\}
\end{gather}
and $Z_{A}$ is a set of points in $\mathbb{C}$ that is determined explicitly by the poles of the recursively defined family of symbols
$$
g_{0}=f_{0}^{-1}, \quad g_{k}=-(T^{-k}f_{0}^{-1})\sum_{j=0}^{k-1}(T^{-j}f_{k-j})g_{j}, \quad k\in\{1,...,\mu-1\}, \quad \mu>1,
$$
where 
$$
f_{\nu}(\lambda)=\frac{1}{\nu!}\sum_{j=0}^{\mu}(\partial_{x}^{\nu}a_{j})(0)\lambda^{j}, \quad \nu\in\{0,...,\mu-1\}, \quad \lambda\in\mathbb{C},
$$
and $T^{\sigma}$, $\sigma\in\mathbb{R}$, denotes the action $(T^{\sigma}f)(\lambda)=f(\lambda+\sigma)$ (see, e.g., \cite[Section 3]{SS1} or \cite[(2.7)-(2.8)]{SS}). The logarithmic exponents $\eta$ in the description of $\mathcal{E}_{A,\gamma}$ are related to the orders of the above poles.

As a particular example of a cone differential operator, we consider the Laplacian associated to \eqref{metric}. On $(0,1)\times\partial\mathcal{B}$ it has the conically degenerate form
$$
\Delta=\frac{1}{x^{2}}\Big((x\partial_{x})^{2}+\Big(n-1+\frac{x\partial_{x}\det[h(x)]}{2\det[h(x)]}\Big)(x\partial_{x})+\Delta_{h(x)}\Big),
$$
where $\Delta_{h(x)}$ is the Laplacian on $\partial\mathcal{B}$ induced by the metric $h(x)$. $\Delta$ is a second order $\mathbb{B}$-elliptic cone differential operator whose conormal symbol is given by
\begin{gather}\label{consymb}
\sigma_{M}(\Delta)(\lambda) = \lambda^{2}-(n-1)\lambda + \Delta_{h(0)}, \quad \lambda\in\mathbb{C}.
\end{gather}
Therefore, if $\{\lambda_{j}\}_{j\in\mathbb{N}\cup\{0\}}$ are the eigenvalues of $\Delta_{h(0)}$, the poles of $(\sigma_{M}(\Delta)(\cdot))^{-1}$ coincide with the set
$$
\Big\{\frac{n-1}2\pm \sqrt{\Big(\frac{n-1}2\Big)^{2}-\lambda_j}\, |\, j\in\mathbb{N}\cup\{0\}\Big\}.
$$
If $\gamma\in(\frac{n-3}{2},\frac{n+1}{2})$, then pole zero is always contained in the strip \eqref{strip}. In this case, denote by $\mathbb{C}_{\omega}$ the subspace of $\mathcal{E}_{\Delta,\gamma}$ corresponding to $\rho=\eta=0$ and $c|_{\partial\mathcal{B}_{j}}=c_{j}$, $c_{j}\in\mathbb{C}$, $j\in\{1,...,\ell\}$, endowed with the norm $\|\cdot\|_{\mathbb{C}_{\omega}}$ defined by $c\mapsto \|c\|_{\mathbb{C}_{\omega}}=(\sum_{j=1}^{\ell}|c_{j}|^{2})^{\frac{1}{2}}$, i.e. $\mathbb{C}_{\omega}$ consists of smooth functions that are locally constant close to the boundary. Under some further restriction on the weight $\gamma$, such a realization satisfies the property of maximal $L^q$-regularity.

\begin{theorem}[{\rm \cite[Theorem 4.1]{Ro2} or \cite[Theorem 6.7]{SS1}}]\label{RsecD}
Let $p\in(1,\infty)$, $s\geq0$ and 
$$
\gamma\in\Big(\frac{n-3}2,\min\Big\{-1+\sqrt{\Big(\frac{n-1}{2}\Big)^{2}-\lambda_{1}} ,\frac{n+1}{2}\Big\}\Big),
$$
where $\lambda_{1}$ is the greatest non-zero eigenvalue of the boundary Laplacian $\Delta_{h(0)}$. Consider the closed extension $\underline{\Delta}_{s}$ of $\Delta$ in 
$$
X_{0}^{s}=\mathcal{H}_{p}^{s,\gamma}(\mathbb{B}) \quad \text{with domain} \quad X_{1}^{s}=\mathcal{D}(\underline{\Delta}_{s})=\mathcal{D}(\underline{\Delta}_{s,\min})\oplus\mathbb{C}_{\omega}=\mathcal{H}_{p}^{s+2,\gamma+2}(\mathbb{B})\oplus\mathbb{C}_{\omega}.
$$ 
Then, for any $\theta\in[0,\pi)$ there exists some $c>0$ such that $c-\underline{\Delta}_{s}$ is $R$-sectorial of angle $\theta$.
\end{theorem}

\begin{remark}
Maximal $L^q$-regularity for a general $\mathbb{B}$-elliptic cone differential operator can be obtained by \cite[Theorem 5.2]{SS1} or \cite[Theorem 4.3]{SS}.
\end{remark}

\subsection*{Integer powers of the Laplacian}

Let $k\in\mathbb{N}$, $k\geq2$, and let $\underline{\Delta}_{s}$ be the Laplacian from Theorem \ref{RsecD}. We have
\begin{gather}\label{DD22}
\mathcal{D}(\underline{\Delta}_{s}^{k})=\mathcal{D}(\underline{\Delta}_{s,\min}^{k})\oplus\mathbb{C}_{\omega}\oplus\mathcal{F}_{k},
\end{gather}
where the integer powers are defined as usual by \eqref{DAk}.

The minimal domain satisfies 
$$
\mathcal{H}_{p}^{s+2k,\gamma+2k}(\mathbb{B})\hookrightarrow \mathcal{D}(\underline{\Delta}_{s,\min}^{k})\hookrightarrow \mathcal{H}_{p}^{s+2k,\gamma+2k-\varepsilon}(\mathbb{B})
$$ 
for all $\varepsilon>0$. Therefore, if $s+2k>\frac{n+1}{p}$ then, due to \cite[Lemma 3.2]{RS3}, $\mathcal{D}(\underline{\Delta}_{s,\min}^{k})\hookrightarrow C(\mathbb{B}^{\circ})$ and for any $u\in\mathcal{D}(\underline{\Delta}_{s,\min}^{k})$ in local coordinates $(x,y)\in(0,1)\times\partial\mathcal{B}$ near the boundary we have
\begin{gather}\label{DD33}
|u(x,y)|\leq L x^{\gamma+2k-\frac{n+1}{2}-\varepsilon}\|u\|_{\mathcal{H}_{p}^{s+2k,\gamma+2k-\varepsilon}(\mathbb{B})} \quad \text{for all} \quad \varepsilon>0,
\end{gather}
with some constant $L>0$ depending only on $\mathbb{B}$ and $p$. 

$\mathcal{F}_{k}$ is an asymptotics space part as in \eqref{dommaxexp}. The set of exponents $\rho$ describing $\mathcal{F}_{k}$ coincides with the finite set of points $Q_{k}\subseteq Z_{k}\cap S_{k}$, where
$$
S_{k}=\Big\{\lambda\in\mathbb{C}\, |\, \mathrm{Re}(\lambda)\in \Big[\frac{n+1}{2}-\gamma-2k,\frac{n+1}{2}-\gamma-2\Big)\Big\} 
$$
and $Z_{k}$ is a set of points in $\mathbb{C}$ that is determined explicitly by $\partial\mathcal{B}$, the family of metrics $h(\cdot)$ from \eqref{metric} and $k$. Moreover, $Q_{k}\subseteq Q_{k-1}\cup V_{k}$ with $Q_{1}=\emptyset$ and a finite set of points $V_{k}$ satisfying
$$
V_{k}\subset \Big\{\lambda\in\mathbb{C}\, |\, \mathrm{Re}(\lambda)\in \Big[\frac{n+1}{2}-\gamma-2k,\frac{n+1}{2}-\gamma-2(k-1)\Big]\Big\}.
$$
Therefore, we also have $\mathcal{F}_{k}\subset \mathcal{H}_{p}^{s+2,\gamma+2}(\mathbb{B})$. 

In particular if the metric $h$ is independent of $x$ when $x$ is close to $0$, then $Q_{k}$ is a subset of 
\begin{gather}
S_{k}\cap \bigcup_{\nu\in\{0,...,k-1\}} \bigcup_{ \lambda_{j}\in\sigma(\Delta_{h(0)})} \Big\{-2\nu+\frac{n-1}2\pm \sqrt{\Big(\frac{n-1}2\Big)^2-\lambda_{j}}\Big\},
\end{gather}
and
\begin{gather}\label{DD23}
\mathcal{F}_{k}=\bigoplus_{\rho\in Q_{k}}\mathcal{F}_{k,\rho}.
\end{gather}
Here, for each $\rho$, $\mathcal{F}_{k,\rho}$ is a finite dimensional space consisting of linear combinations of $C^{\infty}(\mathbb{B}^\circ)$ functions that vanish on $\mathcal{B}\backslash([0,1)\times\partial\mathcal{B})$ and in local coordinates $(x,y)$ on the collar part they are of the form $\omega(x)c(y)x^{-\rho}\log^{\eta}(x)$ with $c\in C^{\infty}(\partial\mathbb{B})$ and $\eta\in\{0,...,\eta_{\rho}\}$, where $\eta_{\rho}\in\mathbb{N}\cup\{0\}$ is the order of $\rho$ as a pole of $(\sigma_{M}(\Delta^{k})(\cdot))^{-1}$, where 
$$
\mathbb{C}\ni\lambda \mapsto \sigma_{M}(\Delta^{k})(\lambda)=\prod_{\nu\in\{k-1,...,0\}} \sigma_{M}(\Delta)(2\nu+\lambda) \in \mathcal{L}(H_{2}^{2k}(\partial\mathbb{B}),H_{2}^{0}(\partial\mathbb{B}))
$$
with $\sigma_{M}(\Delta)(\cdot)$ given by \eqref{consymb}.

\section{Applications}

\subsection{The porous medium equation on manifolds with conical singularities}

The {\em porous medium equation} (PME) is the parabolic diffusion equation 
\begin{eqnarray}\label{e1}
u'(t)-\Delta(u^{m}(t))&=&0,\quad t>0,\\\label{e2}
u(0)&=&u_{0},
\end{eqnarray}
where the scalar function $u$ is a density distribution, $\Delta$ is a (negative) Laplacian and $m>0$ is a fixed parameter. The evolution described by the above equation, models the flow of a gas in a porous medium. The problem \eqref{e1}-\eqref{e2} has been extensively studied in various domains and in many aspects. For a detailed introduction to the theory of PME we refer to V\'azquez \cite{Va}. 

Concerning the case of manifolds with conical singularities, in \cite[Theorem 1.1]{RS4} it has been shown smoothness and long time existence for the solutions of PME. By changing variables in \eqref{e1}-\eqref{e2} we obtain the equivalent problem 
\begin{eqnarray}\label{v1}
w'(t)-mw^{\frac{m-1}{m}}(t)\Delta(w(t))&=&0,\quad t>0,\\\label{v2}
w(0)&=&w_{0}.
\end{eqnarray}
Similarly to \cite[Theorem 1.1]{RS4}, we have the following.

\begin{proposition}[Solutions of \eqref{v1}-\eqref{v2}]\label{Prop1}
Denote by $\lambda_{1}$ the greatest nonzero eigenvalue of $\Delta_{h(0)}$ and choose $p,q\in(1,\infty)$ sufficiently large such that 
$$
\frac{2}{q}<-\frac{n-1}{2}+\sqrt{\Big(\frac{n-1}2\Big)^2-\lambda_1} \quad \text{and} \quad \frac{n+1}{p}+\frac{2}{q}<1.
$$
Let
$$
s_{0}>\max\Big\{-1+\frac{n+1}{p}+\frac{2}{q},-\frac{2}{q}\Big\},\quad \gamma\in \Big(\frac{n-3}{2}+\frac{2}{q},\min\Big\{-1+\sqrt{\Big(\frac{n-1}2\Big)^2-\lambda_1} ,\frac{n+1}2\Big\}\Big)
$$
and
$$
u_{0}\in(\mathcal{H}_{p}^{s_{0}+2,\gamma+2}(\mathbb{B})\oplus\mathbb{C}_{\omega},\mathcal{H}_{p}^{s_{0},\gamma}(\mathbb{B}))_{\frac{1}{q},q}\hookleftarrow \bigcup_{\nu>0}\mathcal{H}_{p}^{s_{0}+2-\frac{2}{q}+\nu,\gamma+2-\frac{2}{q}+\nu}(\mathbb{B})\oplus\mathbb{C}_{\omega}
$$
satisfying $u_{0}\geq\alpha$, for some $\alpha>0$. Then:\\
{\bf (i)} For each $T>0$ there exists a unique 
\begin{gather}\label{lpregw}
w\in W^{1,q}(0,T;\mathcal{H}_{p}^{s_{0},\gamma}(\mathbb{B})) \cap L^{q}(0,T;\mathcal{H}_{p}^{s_{0}+2,\gamma+2}(\mathbb{B})\oplus\mathbb{C}_{\omega})
\end{gather}
solving \eqref{v1}-\eqref{v2} with initial value $w_{0}=u_{0}^{m}$.\\
{\bf (ii)} In addition we have that
\begin{eqnarray}\label{wreg1}
 && w \in C([0,T], (\mathcal{H}_{p}^{s_{0}+2,\gamma+2}(\mathbb{B})\oplus\mathbb{C}_{\omega},\mathcal{H}_{p}^{s_{0},\gamma}(\mathbb{B}))_{\frac{1}{q},q})\\\label{wreg2}
&&\cap \, C([0,T];\mathcal{H}_{p}^{s_{0}+2-\frac{2}{q}-\varepsilon,\gamma+2-\frac{2}{q}-\varepsilon}(\mathbb{B})\oplus\mathbb{C}_{\omega})\\\label{wreg3}
&&\cap\, C^{\delta}([0,T];\mathcal{H}_{p}^{s_{0}+2-\frac{2}{q}-2\delta,\gamma+2-\frac{2}{q}-2\delta}(\mathbb{B})\oplus\mathbb{C}_{\omega}) \cap C^{\delta}((0,T];\mathcal{H}_{p}^{s,\gamma+2-\frac{2}{q}-2\delta}(\mathbb{B})\oplus\mathbb{C}_{\omega})\\\label{wreg4}
&&\cap\, C^{1+\delta}([0,T];\mathcal{H}_{p}^{s_{0}-\frac{2}{q}-2\delta,\gamma-\frac{2}{q}-2\delta}(\mathbb{B}))\cap C^{1+\delta}((0,T];\mathcal{H}_{p}^{s,\gamma-\frac{2}{q}-2\delta}(\mathbb{B}))\\\label{wreg5}
&&\cap \, C^{k}((0,T];\mathcal{H}_{p}^{s,\gamma-2(k-1)}(\mathbb{B})) \cap C((0,T];\mathcal{H}_{p}^{s,\gamma+2}(\mathbb{B})\oplus\mathbb{C}_{\omega})
\end{eqnarray}
for all $k\in\mathbb{N}$, $s>0$, $\varepsilon>0$ and 
\begin{gather}\label{deltachoice}
\delta\in\Big(0,\frac{1}{2}\min\Big\{2-\frac{n+1}{p}-\frac{2}{q},\gamma-\frac{n-3}{2}-\frac{2}{q}\Big\}\Big).
\end{gather}
{\bf (iii)} For each $T>0$ we have $w=u^{m}$ on $[0,T]\times\mathbb{B}$, where $u$ is the solution of \eqref{e1}-\eqref{e2} with initial data $u_{0}$ given by \cite[Theorem 1.1]{RS4}. In particular, if $c_{0}\leq u_{0}^{m}\leq c_{1}$ on $\mathbb{B}$ for suitable constants $c_{0},c_{1}>0$, then also $c_{0}\leq w\leq c_{1}$ on $[0,T]\times\mathbb{B}$ for each $T>0$.
\end{proposition}
\begin{proof} 
First note that the assumption $u_{0}\geq\alpha$ makes sense due to \cite[Lemma 5.2]{RS3} and \cite[Lemma 3.2]{RS3}. By \eqref{intemb}, \cite[Lemma 5.2]{RS3}, \cite[(5.29)]{RS4} and \cite[(5.32)]{RS4}, for each $T>0$ there exists a unique $w$ as in \eqref{lpregw} solving \eqref{v1}-\eqref{v2} with initial value $w_{0}=u_{0}^{m}$, which moreover satisfies \eqref{wreg1}, \eqref{wreg2} and \eqref{wreg3} for all $s>0$, $\varepsilon>0$ and $\delta$ as in \eqref{deltachoice}, and also \eqref{wreg5} with $k=1$. Moreover, due to \cite[Theorem 1.1]{RS4}, \cite[Remark 2.12]{RS4} and the comparison principle \cite[Theorem 4.3]{RS4} we have that $w=u^{m}$ on $[0,T]\times\mathbb{B}$. In addition, by \cite[Lemma 2.5]{RS4} $w^{\frac{m-1}{m}}$ satisfies \eqref{wreg2} and the right hand side of \eqref{wreg5}. Therefore, by \eqref{v1} and \cite[Corollary 3.3]{RS3} $w$ satisfies \eqref{wreg4}.

By \cite[Theorem 1.1]{RS4} we have that $\partial_{t}^{2}u\in C((0,T];\mathcal{H}_{p}^{s,\gamma-2}(\mathbb{B}))$ and for $s$ sufficiently large
$$
u^{m-1}, u^{m-2}\in C((0,T];\mathcal{H}_{p}^{s,\gamma+2}(\mathbb{B})\oplus\mathbb{C}_{\omega})
$$
due to \cite[Lemma 2.5]{RS4}. Hence, if $\chi$ is a smooth function on $\mathbb{B}^{\circ}$ such that $\chi\geq\frac{1}{4}$ on $\mathcal{B}\backslash([0,\frac{1}{2})\times\partial\mathcal{B})$ and $\chi=x^{2}$ on $[0,\frac{1}{2})\times\partial\mathcal{B}$, then by
$$
\partial_{t}^{2}w=mu^{m-1}\partial_{t}^{2}u+m(m-1)u^{m-2}(\chi\partial_{t}u)(\chi^{-1}\partial_{t}u)
$$ 
and \cite[Lemma 2.5]{RS4} we deduce that $w\in C^{2}((0,T];\mathcal{H}_{p}^{s,\gamma-2}(\mathbb{B}))$. The assertion then follows by iteration. 
\end{proof}

The above result combined with Theorem \ref{expth} provides the following expansion for the solutions of PME. 

\begin{proposition}[Expanding solutions of PME]\label{pmexp}
Let $p$, $q$, $s_{0}$, $\gamma$, $u_{0}$ be chosen as in \cite[Theorem 1.1]{RS4} and let $u$ be the unique solution of the porous medium equation \eqref{e1}-\eqref{e2} according to this theorem. For any $t_{1},\mu,\xi,\varepsilon>0$ there exist $t_{2}>t_{1}$ and 
\begin{eqnarray*}
\lefteqn{v\in C([t_{1},t_{2}];C(\mathbb{B}))\cap \bigcap_{s,\nu>0} \mathcal{A}(\Lambda_{t_{1}}(\pi/2);\mathcal{D}(\underline{\Delta}_{s}^{2}))}\\
&&\cap \, W^{1,q}(t_{1},t_{2};\mathcal{H}_{p}^{s,\gamma}(\mathbb{B})) \cap L^{q}(t_{1},t_{2};\mathcal{H}_{p}^{s,\gamma+2}(\mathbb{B})\oplus\mathbb{C}_{\omega}) \cap C([t_{1},t_{2}];\mathcal{H}_{p}^{s,\gamma+2-\frac{2}{q}-\nu}(\mathbb{B})\oplus\mathbb{C}_{\omega})
\end{eqnarray*}
such that 
\begin{eqnarray*}
\lefteqn{\|u^{m}-v\|_{W^{1,q}(t_{1},t_{2};\mathcal{H}_{p}^{\mu,\gamma}(\mathbb{B})) \cap L^{q}(t_{1},t_{2};\mathcal{H}_{p}^{\mu,\gamma+2}(\mathbb{B})\oplus\mathbb{C}_{\omega})}}\\
&&+\|u^{m}-v\|_{C([t_{1},t_{2}];\mathcal{H}_{p}^{\mu,\gamma+2-\frac{2}{q}-\xi}(\mathbb{B})\oplus\mathbb{C}_{\omega})}+\|u^{m}-v\|_{C([t_{1},t_{2}];C(\mathbb{B}))}<\varepsilon
\end{eqnarray*}
and $u^{m}(t_{1})=v(t_{1})$. Here $\mathcal{D}(\underline{\Delta}_{s}^{2})$ is described by \eqref{DD22} and provides the required expansion for $u^{m}$. Moreover, if $t_{1}\in [\tau_{1},\tau_{2}]$, $0<\tau_{1}<\tau_{2}<\infty$, then $t_{2}-t_{1}$ is only determined by $u$, $\mu$, $\xi$, $\varepsilon$, $\tau_{1}$ and $\tau_{2}$. 
\end{proposition}
\begin{proof}
Let $s>\frac{n+1}{p}$. We apply Theorem \ref{expth} to \eqref{v1}-($w(0)=u^{m}(\tau_{1})$), with $X_{0}=X_{0}^{s}$, $X_{1}=X_{1}^{s}$ and $A(\cdot)=-m(\cdot)^{\frac{m-1}{m}}\underline{\Delta}_{s}$. For any $T>\tau_{1}$, let $w$ be the unique solution of \eqref{v1}-\eqref{v2} on $[0,T]\times\mathbb{B}$ according to Proposition \ref{Prop1}. By \cite[Lemma 2.5]{RS4} we have that $\eta=w^{\frac{m-1}{m}}\in C([\tau_{1},\tau_{2}];X_{1}^{s})$. Also, due to \cite[Corollary 3.3]{RS3}, elements in $X_{1}^{s}$ act by multiplication as bounded maps on $X_{0}^{s}$, so that for each $\xi\in [\tau_{1},\tau_{2}]$, $A(w(\xi)):X_{1}^{s}\rightarrow X_{0}^{s}$ is well defined and furthermore $A(w(\cdot))\in C([\tau_{1},\tau_{2}];\mathcal{L}(X_{1}^{s},X_{0}^{s}))$. Finally, by \cite[Theorem 6.1]{RS3}, for each $\xi\in [\tau_{1},\tau_{2}]$ and each $\theta\in [0,\pi)$ there exists a $c>0$ such that $A(w(\xi))+c\in\mathcal{R}(\theta)$.

We conclude that there exist $t_{2}\in (t_{1}, \tau_{2}]$ and $v$ having the required properties, where we have also used \cite[Lemma 3.2]{RS3} and \cite[Lemma 5.2]{RS3}. In particular, $v(t)\in \mathcal{D}((\eta(t_{1})\underline{\Delta}_{s})^{2})$, $t\in(t_{1},t_{2})$, implies that $v(t)\in \mathcal{D}(\underline{\Delta}_{s})$ and $\eta(t_{1})\underline{\Delta}_{s}v(t)\in \mathcal{D}(\underline{\Delta}_{s})$. Moreover, by Proposition \ref{Prop1} we have $c_{0}^{\frac{m-1}{m}}\leq \eta\leq c_{1}^{\frac{m-1}{m}}$ on $[\tau_{1},\tau_{2}]\times\mathbb{B}$, for certain $c_{0},c_{1}>0$. Therefore, by \cite[Lemma 2.5]{RS4} we deduce that $\Delta v(t)\in \mathcal{D}(\underline{\Delta}_{s})$, which implies that $v(t)\in\mathcal{D}(\underline{\Delta}_{s}^{2})$, $t\in(t_{1},t_{2})$.
\end{proof}

\begin{remark}
In the case of the heat equation, i.e. when $m=1$, Proposition \ref{pmexp} can be improved to \cite[Theorem 4.3]{Ro3}. In this situation, the asymptotics space expansion of $u$ can be chosen arbitrary long, i.e. $u$ is analytic in time with values in $\mathcal{D}(\underline{\Delta}_{s}^{k})$ from \eqref{DD22}, for each $k\in\mathbb{N}$.
\end{remark}

\subsection{The Swift-Hohenberg equation on manifolds with conical singularities}

We consider the following problem
\begin{eqnarray}\label{CH1}
u'(t)+(\Delta+1)^{2}u(t)&=&V(u(t),t), \quad t\in(0,T),\\ \label{CH2}
u(0)&=&u_{0},
\end{eqnarray}
called {\em Swift-Hohenberg equation} (SHE), where $u$ is a scalar field, $\Delta$ is a (negative) Laplacian, $V(x,s)$ is a polynomial in $x$ with $s$-dependent complex valued coefficients that are locally Lipschitz continuous in $\mathbb{R}$ and $T>0$ is finite. The above equation has various applications in physics; it can model e.g. thermally convecting fluid flows \cite{SH}, cellular flows \cite{PM} as well as phenomena in optical physics \cite{TGM}. It is also well known for its pattern formation under evolution.

In \cite{Ro2} the problem \eqref{CH1}-\eqref{CH2} was considered on manifolds with conical singularities. In \cite[Theorem 4.1]{Ro2} it has been shown existence and maximal regularity of the short time solution. In addition, in \cite[Theorem 4.4]{Ro2} a necessary and sufficient condition was found such that the above solution exists for all times. By combining the results in \cite{Ro2} with \cite[Theorem 3.1]{RS4} we obtain the following smoothness for the solutions of \eqref{CH1}-\eqref{CH2}.

\begin{proposition}[Smoothness for solutions of SHE]
Let $p$, $q$, $s$, $\gamma$, $u_{0}$ be chosen as in \cite[Theorem 4.1]{Ro2} and let $u$ be the unique solution of the Swift-Hohenberg equation \eqref{CH1}-\eqref{CH2} on $[0,T]\times\mathbb{B}$, for some $T>0$, according to this theorem. Then, in addition to \cite[(4.21)-(4.22)]{Ro2} we have
$$
u\in \bigcap_{\nu\geq0} C^{1}((0,T);\mathcal{H}_{p}^{\nu,\gamma}(\mathbb{B}))\cap C((0,T);\mathcal{D}(\underline{\Delta}_{\nu}^{2})),
$$
where the bi-Laplacian domain is also described by \eqref{DD22}.
\end{proposition}
\begin{proof}
By \cite[Theorem 4.1]{Ro2} it suffices to show that for each $\delta\in (0,T)$ we have
\begin{gather}\label{instreg}
u\in \bigcap_{\nu\geq0}W^{1,q}(\delta,T;\mathcal{H}_{p}^{\nu,\gamma}(\mathbb{B}))\cap L^{q}(\delta,T;\mathcal{D}(\underline{\Delta}_{\nu}^{2})).
\end{gather}
To this end we apply \cite[Theorem 3.1]{RS4} to \eqref{CH1}-\eqref{CH2}. By \cite[(4.23)]{Ro2} for each $\nu\geq0$ and each $\theta\in[0,\pi)$ there exists a $c>0$ such that the operator 
\begin{gather}\label{maxregforbilap}
(\underline{\Delta}_{\nu}+1)^{2}+c:\mathcal{D}(\underline{\Delta}_{\nu}^{2})\rightarrow \mathcal{H}_{p}^{\nu,\gamma}(\mathbb{B}) 
\end{gather}
belongs to $\mathcal{R}(\theta)$, so that by Theorem \ref{KWth} it has maximal $L^{q}$-regularity. Moreover, by \cite[(4.30)]{Ro2}
\begin{gather}\label{Vreg}
V(v(\cdot),\cdot)\in \bigcap_{\varepsilon>0}C([0,T];\mathcal{H}_{p}^{\nu+2-\frac{2}{q}-\varepsilon,\gamma+2-\frac{2}{q}-\varepsilon}(\mathbb{B})\oplus\mathbb{C}_{\omega})
\end{gather}
whenever
$v\in C([0,T]; (\mathcal{D}(\underline{\Delta}_{\nu}^{2}),\mathcal{H}_{p}^{\nu,\gamma}(\mathbb{B}))_{\frac{1}{q},q}$.

Fix $\rho>\frac{1}{2}\max\{\frac{1}{q-1},\frac{1}{2}\}$ and consider the Banach scales $Y_{0}^{j}=\mathcal{H}_{p}^{s+\frac{j}{\rho q},\gamma}(\mathbb{B})$ and $Y_{1}^{j}=\mathcal{D}(\underline{\Delta}_{s+\frac{j}{\rho q}}^{2})$, $j\in\mathbb{N}\cup\{0\}$. By \cite[Lemma 7.2]{RSS} or \cite[Theorem 3.3]{Ro3} for each $j\in\mathbb{N}\cup\{0\}$ we have $Y_{1}^{j}\hookrightarrow(Y_{1}^{j+1},Y_{0}^{j+1})_{\frac{1}{q},q}$. Moreover, by \eqref{Vreg}, $V(v(\cdot),\cdot)\in L^{q}(0,T;Y_{0}^{j+1})$ whenever $v\in C([0,T];(Y_{1}^{j},Y_{0}^{j})_{\frac{1}{q},q})$, $j\in\mathbb{N}\cup\{0\}$. Therefore, by \cite[Theorem 3.1]{RS4} with $Z=(Y_{1}^{0},Y_{0}^{0})_{\frac{1}{q},q}$ and $F=V$ we obtain \eqref{instreg}.
\end{proof}

By combining Theorem \ref{expth}, \eqref{intemb}, \eqref{instreg}, \eqref{maxregforbilap}, \eqref{Vreg}, \cite[(4.24)]{Ro2} and \cite[(4.31)]{Ro2} we obtain the following. 

\begin{proposition}[Expanding solutions of SHE]\label{shexp}
Let $p$, $q$, $s$, $\gamma$, $u_{0}$ be chosen as in \cite[Theorem 4.1]{Ro2} and let $u$ be the unique solution of the Swift-Hohenberg equation \eqref{CH1}-\eqref{CH2} on $[0,T]\times\mathbb{B}$, for some $T>0$, according to this theorem. Then, for any $t_{1},\mu,\xi,\varepsilon>0$ there exist $t_{2}>t_{1}$ and 
\begin{eqnarray*}
\lefteqn{v\in C([t_{1},t_{2}];C(\mathbb{B}))\cap \bigcap_{r,k,\nu>0} \mathcal{A}(\Lambda_{t_{1}}(\pi/2);\mathcal{D}(\underline{\Delta}_{r}^{k}))}\\
&&\cap \, W^{1,q}(t_{1},t_{2};\mathcal{H}_{p}^{r,\gamma}(\mathbb{B})) \cap L^{q}(t_{1},t_{2};\mathcal{D}(\underline{\Delta}_{r}^{2})) \cap C([t_{1},t_{2}];\mathcal{H}_{p}^{r,\gamma+2-\frac{2}{q}-\nu}(\mathbb{B})\oplus\mathbb{C}_{\omega})
\end{eqnarray*}
such that 
\begin{eqnarray*}
\lefteqn{\|u-v\|_{W^{1,q}(t_{1},t_{2};\mathcal{H}_{p}^{\mu,\gamma}(\mathbb{B})) \cap L^{q}(t_{1},t_{2};\mathcal{D}(\underline{\Delta}_{\mu}^{2}))}}\\
&&+\|u-v\|_{C([t_{1},t_{2}];\mathcal{H}_{p}^{\mu,\gamma+2-\frac{2}{q}-\xi}(\mathbb{B})\oplus\mathbb{C}_{\omega})}+\|u-v\|_{C([t_{1},t_{2}];C(\mathbb{B}))}<\varepsilon
\end{eqnarray*}
and $u(t_{1})=v(t_{1})$. Here $\mathcal{D}(\underline{\Delta}_{r}^{k})$, $k\in\mathbb{N}$, is described by \eqref{DD22} and provides the required expansion for $u$. If in addition $t_{1}\in [\tau_{1},\tau_{2}]$, $0<\tau_{1}<\tau_{2}<\infty$, then $t_{2}-t_{1}$ is only determined by $u$, $\mu$, $\xi$, $\varepsilon$, $\tau_{1}$ and $\tau_{2}$, and for each $\alpha\in (0,\frac{1}{2}-\frac{1}{2q})$ we have
$$
\|(u-v)(t_{2})\|_{\mathcal{D}(\underline{\Delta}_{\mu}^{2})}\leq C(t_{2}-t_{1})^{\alpha}
$$
for some constant $C>0$ that only depends on $\alpha$, $q$, $u$, $\mu$, $\tau_{1}$ and $\tau_{2}$.
\end{proposition}

\begin{remark}
The fact that the approximation $v$ of $u$ in Proposition \ref{shexp} has values in $\mathcal{D}(\underline{\Delta}_{s}^{k})$ for arbitrary $k\in\mathbb{N}$ allows us to choose an arbitrary long expansion for $u$ according to \eqref{DD22}. Furthermore, $q$ in Proposition \ref{shexp} can be chosen arbitrary large and hence $\alpha$ arbitrary close to $1/2$.
\end{remark}

\subsection{The case of closed manifolds}

Let $\mathcal{M}$ be a smooth, closed and connected $(n+1)$-dimensional Riemannian manifold, endowed with a Riemannian metric $f$ and let $\mathbb{S}^{n}$ be the unit sphere $\{z\in\mathbb{R}^{n+1}\, |\, |z|=1\}$. Fix a point $o\in \mathcal{M}$ and denote by $x=d(o,z)$ the geodesic distance between $o$ and $z\in \mathcal{M}\setminus\{o\}$, where $d$ is the metric distance induced by $f$. There exists an $r>0$ such that $(x,y)\in (0,r)\times \mathbb{S}^{n}$ are local coordinates around $o$ and moreover the metric in these coordinates becomes 
$$
f=dx^{2}+x^{2}f_{\mathbb{S}^{n}}(x),
$$
where $x\mapsto f_{\mathbb{S}^{n}}(x)$ is a smooth family of Riemannian metrics on $\mathbb{S}^{n}$ (see, e.g., \cite[Lemma 5.5.7]{Pa}). Assuming that $f_{\mathbb{S}^{n}}(\cdot)$ is smooth up to $x=0$ and also does not degenerate up to this point, we can regard $((\mathcal{M}\backslash\{o\})\cup(\{0\}\times \mathbb{S}^{n}),f)$ as a conic manifold with one isolated conical singularity at the pole $o$. From this point of view, the results of this section are applied to the porous medium equation and the Swift-Hohenberg equation on $(\mathcal{M},f)$ respectively. Note that in both cases the solution is well defined on $\mathcal{M}$ since it is a time dependent constant on $\{0\}\times \mathbb{S}^{n}$.

\end{document}